\documentclass{amsart}
\usepackage{fullpage,amssymb,epic,eepic,epsfig,amsmath,amsaddr}
\usepackage[utf8]{inputenc}

\theoremstyle{plain}

\newtheorem{theorem}{Theorem}
\newtheorem{proposition}{Proposition}[section]
\newtheorem{lemma}[proposition]{Lemma}

\newtheorem{corollary}[proposition]{Corollary}

\theoremstyle{definition}
\newtheorem{definition}[proposition]{Definition}

\theoremstyle{remark}

\def\printname#1{
	\if\draft y
		\smash{\makebox[0pt]{\hspace{-0.5in}
			\raisebox{8pt}{\tt\tiny #1}}}
	\fi
}

\newlength{\standardunitlength}
\catcode`\@=11
\long\def\@makecaption#1#2{%
     \vskip 10pt

\setbox\@tempboxa\hbox{
       \small\sf{\bfcaptionfont #1. }\ignorespaces #2}%
     \ifdim \wd\@tempboxa >\captionwidth {%
         \rightskip=\@captionmargin\leftskip=\@captionmargin
         \unhbox\@tempboxa\par}%
       \else
         \hbox to\hsize{\hfil\box\@tempboxa\hfil}%
     \fi}
\font\bfcaptionfont=cmssbx10 scaled \magstephalf
\newdimen\@captionmargin\@captionmargin=2\parindent
\newdimen\captionwidth\captionwidth=\hsize
\catcode`\@=12

\def\C{\mathcal C}

\def\D{\Delta}

\def\T{\mathcal T}

\def\T{\mathcal T}

\def\s{\sigma}

\def\s{\sigma}

\def\Per{\mathrm{per}}

\def\sign{\mathrm{sign}}

\begin{document}

\title[Binary Linear Codes, Dimers and Hypermatrices]{Binary Linear Codes, Dimers and Hypermatrices}

\author{Martin Loebl}
\address{Dept.~of Applied Mathematics\\
Charles University\\
Praha\\
Czech Republic.}
\email{loebl@kam.mff.cuni.cz}

\author{Pavel Rytir}
\address{The City College, The City University of New York\\
New York\\
 U.S.A.}
\email{rytirpavel@gmail.com}

\date{\today}
\thanks{The first author was artially supported by the Czech Science Foundation under contract number ~P202-13-21988S}. 

\begin{abstract}
We show that the weight enumerator of any binary linear code is equal to 
the permanent of a 3-dimensional hypermatrix (3-matrix). We also show that each permanent is a determinant of a 3-matrix. As an application we write the dimer partition function of a finite 3-dimensional cubic lattice as the determinant of the vertex-adjacency 3-matrix of a 2-dimensional simplicial complex which preserves the natural embedding of the cubic lattice.
\end{abstract}

\maketitle



\section{Introduction}
\label{sec.int}
The {\em Kasteleyn method} is a way how to calculate the Ising partition function on a finite graph $G$. It goes as follows. We first realize that the Ising partition function is equivalent to a
multivariable weight enumerator of the cut space of $G$. We modify $G$ to graph $G'$ so that 
this weight enumerator is equal to the generating function of the perfect matchings of $G'$ (the dimer partition function of $G$).
Such generating functions are hard to calculate. In particular, if $H$ is a bipartite graph then the generating function of the perfect matchings of $H$ is equal to the permanent of the biadjacency
matrix of $H$. If however this permanent may be turned into the determinant of a modified
matrix then the calculation can be successfully carried over since the determinants may be calculated efficiently. Already in 1913 Polya asked for which non-negative matrix $M$
we can change signs of its  entries so that, denoting by $M'$ the resulting matrix, we have $\Per(M)= \det(M')$.  We call these matrices {\em Kasteleyn} after the physicist Kasteleyn
who invented the Kasteleyn method. Kasteleyn proved in 1960's that all biadjacency matrices of the planar bipartite graphs are Kasteleyn. We say that a bipartite graph is Pfaffian if its biadjacency matrix is Kasteleyn. The problem to characterize the Kasteleyn matrices (or equivalently Pfaffian bipartite graphs) was
open until 1993, when Robertson, Seymour and Thomas \cite{RST} found a polynomial recognition
method and a structural description of the Kasteleyn matrices. They showed that the class
of the Kasteleyn matrices is rather restricted and extends only moderately beyond the biadjacency matrices of the
planar bipartite graphs.

Some years ago M.L. suggested to (1) extend the Pfaffian method to weight enumerators of general binary linear codes, and
(2) to use hypermatrices instead of matrices to gain new insight into the Ising and dimer problems for the cubic lattice.

In this paper we show that the weight enumerator of any binary linear code is equal to 
the permanent of the triadjacency 3-matrix of a 2-dimensional simplicial complex. In analogy to the standard
(2-dimensional) matrices we say that a 3-dimensional non-negative matrix $A$ is
 {\em Kasteleyn} if signs of its entries may be changed so that, denoting by $A'$ 
the resulting 3-dimensional matrix, we have $\Per(A)= \det(A')$. We show that in contrast
with the 2-dimensional case the class of Kasteleyn 3-dimensional matrices is rich; 
namely, for each 2-dimensional non-negative matrix $M$ there is a   3-dimensional non-negative Kasteleyn matrix $A$ so that $\Per(M)= \Per(A)$.  

Summarising, we have the following applications for the basic 3-dimensional statistical physics models.
\begin{itemize}
\item
We write the partition function of the dimer problem in the cubic lattice as 3-determinant.
\item
We write the partition function of the Ising problem in the cubic lattice as 3-permanent.
\end{itemize}

Using some results of this paper ML showed in \cite{L1} how to write both these partition functions as a single formal product.

\subsection{Basic definitions}
\label{sub.bd}
A linear code $\mathcal{C}$ of length $n$ and dimension $d$ over a field $\mathbb{F}$ is a linear subspace with dimension $d$ of the vector space $\mathbb{F}^n$. Each vector in $\mathcal{C}$ is called a codeword.
The weight $w(c)$ of a codeword $c$ is the number of non-zero entries of $c$; if we are given $w(i)$, $i= 1, \ldots, n$
then $w(c)= \sum_{i: c_i= 1}w(i)$.
The weight enumerator of a finite code $\mathcal{C}$ is defined according to the formula
\begin{equation*}
W_\mathcal{C}(x):=\sum_{c\in \mathcal{C}} x^{w(c)}.
\end{equation*}

A simplex $X$ is the convex hull of an affine independent set $V$ in $\mathbb{R}^d$. The dimension of $X$ is $\left\lvert V\right\rvert-1$, denoted by $\dim X$.
The convex hull of any non-empty subset of $V$ that defines a simplex is called a face of the simplex.
A simplicial complex $\Delta$ is a set of simplices fulfilling the following conditions: Every face of a simplex from $\Delta$ belongs to $\Delta$ and the intersection of every two simplices of $\Delta$ is a face of both.
The dimension of $\Delta$ is $\max\left\lbrace\dim X\vert X\in\Delta\right\rbrace$.
Let $\Delta$ be a $d$-dimensional simplicial complex. We define the incidence matrix $I=\left(I_{ij}\right)$ as follows: The rows are indexed by $\left(d-1\right)$-dimensional simplices and the columns by $d$-dimensional simplices. We set
\begin{equation*}
I_{ij}:=\begin{cases}
         1& \text{if }(d-1)\text{-simplex }i\text{ belongs to }d\text{-simplex }j,\\
         0& \text{otherwise}.
        \end{cases}
\end{equation*}
This paper studies 2-dimensional simplicial complexes where each maximal simplex is a triangle. We call them triangular configurations.
We denote the set of vertices of $\Delta$ by $V(\Delta)$, the set of edges by $E(\Delta)$ and the set of triangles by $T(\Delta)$.
The cycle space of $\Delta$ over a field $\mathbb{F}$, denoted $\ker_{\mathbb{F}}\Delta$, is the kernel of the incidence matrix $A$ of $\Delta$ over $\mathbb{F}$, that is
$\{x\vert Ax=_{\mathbb{F}}0\}$.

Let $\Delta$ be a triangular configuration. A matching of $\Delta$ is defined with respect to edges; hence, a matching of $\Delta$ is a subconfiguration $M$ of $\Delta$ such that $t_1\cap t_2$ does not contain an edge for every distinct $t_1,t_2\in T(M)$.
Let $\Delta$ be a triangular configuration. Let $M$ be a matching of $\Delta$. Then the defect of $M$ is the set $E(T)\setminus E(M)$.
The perfect matching of $\Delta$ is a matching with empty defect. We denote the set of all perfect matchings of $\Delta$ by $\mathcal{P}(\Delta)$.
Let $w:T(\Delta)\mapsto\mathbb{R}$ be weights of the triangles of $\Delta$.
The generating function of perfect matchings in $\Delta$ is defined to be $P_\Delta(x)=\sum_{P\in\mathcal{P}(\Delta)} x^{w(P)}$, where $w(P):=\sum_{t\in P} w(t)$.

A triangular configuration $\Delta$ is tripartite if the edges of $\Delta$ can be divided into three disjoint sets $E_1,E_2,E_3$ such that every triangle of $\Delta$ contains edges from all sets $E_1,E_2,E_3$. We call the sets $E_1,E_2,E_3$ tripartition of $\Delta$.

We recall that biadjacency matrix $A(x)=(a_{ij})$ of a bipartite graph $G=(V,W,E)$ is the $\lvert V\rvert\times\lvert W\rvert$ matrix, defined as follows: We set
\begin{equation*}
a_{ij}:=\begin{cases}
         x^{w(e)}& \text{if }v_i\in V,v_j\in W\text{ form an edge }e\text{ with weight }w(e),\\
         0& \text{otherwise}.
        \end{cases}
\end{equation*}

The triadjacency 3-matrix $A_{\Delta}(x)= A(x)=\left(a_{ijk}\right)$ of a tripartite triangular configuration $\Delta$ with tripartition $E_1,E_2,E_3$ is the $\lvert E_1\rvert \times \lvert E_2\rvert\times \lvert E_3\rvert$ three dimensional array of numbers, defined as follows: We set
\begin{equation*}
a_{ijk}:=\begin{cases}
         x^{w(t)}& \text{if }e_i\in E_1,e_j\in E_2,e_k\in E_3\text{ form a triangle }t\text{ with weight }w(t),\\
         0& \text{otherwise}.
        \end{cases}
\end{equation*}

The permanent of a $n\times n\times n$ 3-matrix $A$ is defined to be
\begin{equation*}
\Per(A)=\sum_{\sigma_1,\sigma_2\in S_n}\prod^n_{i=1}a_{i\sigma_1(i)\sigma_2(i)}.
\end{equation*} 

The determinant of a $n\times n\times n$ 3-matrix $A$ is defined to be
\begin{equation*}
\det(A)=\sum_{\s_1,\s_2\in S_n}\sign(\s_1)\sign(\s_2)\prod^n_{i=1}a_{i\sigma_1(i)\sigma_2(i)}.
\end{equation*}

\subsection{Main results}
\label{sub.main}

\begin{theorem}
\label{thm.main1}
 Let $\C$ be a binary linear code. Then there exists a tripartite triangular configuration $\Delta$ and weights $w:T(\Delta)\mapsto\mathbb{R}$ such that: If $A_{\Delta}(x)$ is triadjacency matrix of $\Delta$ then 
$$
\Per(A_{\Delta}(x))= W_{\C}(x).
$$
\end{theorem}

\begin{proof}
 Follows from Theorems~\ref{thm:repr1}, \ref{thm:matching} and \ref{thm.main2} below.
\end{proof}

\begin{theorem}
\label{thm:repr1}
Let $\mathcal{C}$ be a binary linear code. Then there exists a triangular configuration $\Delta$
and weights $w:T(\Delta)\mapsto\mathbb{R}$ such that
$W_\mathcal{C}(x)$ is equal to the weight enumerator of $\ker_{\mathbb{F}_2}\Delta$.
\end{theorem}
\begin{proof}
This follows from Theorem 6 of \cite{rytir2} by setting the weights of all the auxiliary triangles to zero.

\end{proof}

This theorem is extended to linear codes over $GF(p)$, where $p$ is a prime, in \cite{rytir3}.

\begin{theorem}[Rytíř \cite{rytir2}]
\label{thm:matching}
Let $\Delta$ be a triangular configuration with weights $w:T(\Delta)\mapsto\mathbb{R}$. Then there exists a triangular configuration $\Delta'$ and weights $w':T(\Delta')\mapsto\mathbb{R}$ such that $W_{\ker\Delta}(x)=P_{\Delta'}(x)$.
\end{theorem}

\begin{theorem}
\label{thm.main2}
 Let $\Delta$ be a triangular configuration with weights $w:T(\Delta)\mapsto\mathbb{R}$. Then there exists a tripartite triangular configuration $\Delta'$ and weights $w':T(\Delta')\mapsto\mathbb{R}$ such that $P_\Delta(x)=\Per(A_{\Delta'}(x))$ where $A_{\Delta'}(x)$ is the triadjacency matrix of $\Delta'$.
\end{theorem}
\begin{proof}
 Follows directly from Proposition~\ref{prop.trip1} and Proposition~\ref{prop:per} of Section~\ref{sec.construction}.
\end{proof}

\begin{definition}
\label{def.3ka}
We say that an $n\times n\times n$ 3-matrix $A$ is {\em Kasteleyn} if
there is 3-matrix $A'$ obtained from $A$ by changing signs of some entries so that
$\Per (A)= \det(A')$.
\end{definition}

The following Theorem \ref{thm.main3} is proved in Section \ref{sec.3mat}.

\begin{theorem}
\label{thm.main3}
Let $M$ be $n\times n$ matrix. Then one can construct 
 $m\times m\times m$ Kasteleyn 3-matrix $A$ with $m\leq n^2+ 2n$ and $\Per(M)= \Per(A)$.
Moreover, Kasteleyn signing is trivial, i.e., $\Per(A)= \det(A)$, and if $M$ is non-negative then $A$ is non-negative.
\end{theorem}

In the last section, applying Theorem \ref{thm.main3}, we write the dimer partition function of a finite 3-dimensional cubic lattice as the determinant of the vertex-adjacency 3-matrix of a 2-dimensional simplicial complex which preserves the natural embedding of the cubic lattice. We also include the Binet-Cauchy formula for the determinant of a 3-matrix.

\section{Triangular configurations and permanents}
\label{sec.construction}
In this section we prove Theorem~\ref{thm.main2}. We use basic building blocks as in Rytíř~\cite{rytir2}. However, the use is novel and we need to stress the tripartitness of basic blocks. Hence we briefly describe them again.
\subsection{Triangular tunnel}
Triangular tunnel is depicted in Figure~\ref{fig:triantunnel}. An {\em empty triangle} is a set of three edges forming a boundary of a triangle. We call the empty triangles $\{a,b,c\}$ and $\{a',b',c'\}$ ending.
\begin{figure}[h]
 \centering
 \includegraphics[width=320pt]{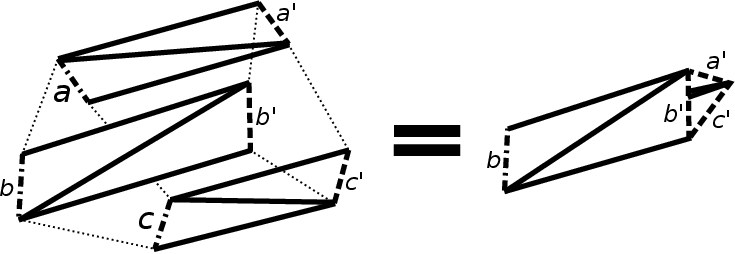}
 \caption{Triangular tunnel}
 \label{fig:triantunnel}
\end{figure}

\begin{proposition}
The triangular tunnel has exactly one matching $M^L$ with defect $\{a,b,c\}$ and exactly one matching $M^R$ with defect $\{a',b',c'\}$.\qed
\end{proposition}

\begin{proposition}
The triangular tunnel is tripartite.
\end{proposition}
\begin{proof}
Follows from Figure~\ref{fig:tunneltrip}.
\begin{figure}
 \centering
 \includegraphics{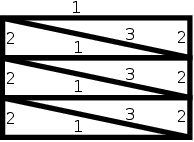}
 \caption{Tunnel tripartition}
 \label{fig:tunneltrip}
\end{figure}

\end{proof}

\subsection{Triangular configuration $S^5$}
Triangular configuration $S^5$ is depicted in Figure~\ref{fig:s5}. Letter "X" denotes empty triangles. We call these empty triangles ending.
\begin{figure}[h]
 \centering
 \includegraphics{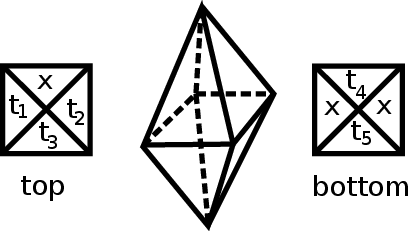}
 \caption{Triangular configuration $S^5$}
 \label{fig:s5}
\end{figure}

\begin{proposition}
Triangular configuration $S^5$ has one exactly perfect matching  and exactly one matching with defect on edges of all empty triangles.
\end{proposition}
\begin{proof}
The unique perfect matching is $\{t_1,t_2,t_4,t_5\}$. We denote it by $M^1(S^5)$. The unique matching with defect on edges of all empty triangles is $\{t_3\}$. We denote it by $M^0(S^5)$.
\end{proof}

\begin{proposition}
Triangular configuration $S^5$ is tripartite. 
\end{proposition}
\begin{proof}
Follows from Figure~\ref{fig:s5colored}.
 \begin{figure}[h]
 \centering
 \includegraphics{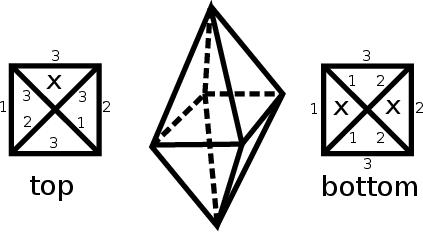}
 \caption{Triangular configuration $S^5$ with partitioning}
 \label{fig:s5colored}
\end{figure}

\end{proof}

\subsection{Matching triangular triangle}
\label{sec:matchtriangle}

\begin{figure}[h]
 \centering
 \includegraphics{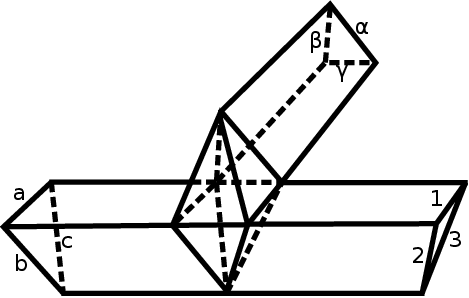}
 \caption{Matching triangular triangle}
 \label{fig:matchtrian}
\end{figure}

The {\em matching triangular triangle} is obtained from the triangular configuration $S^5$ and three triangular tunnels in the following way:
Let $T_1$, $T_2$ and $T_3$ be triangular tunnels. Let $t_1^{T_1},p^{T_1}$; $t_1^{T_2},q^{T_2}$ and $t_1^{T_3},r^{T_3}$ be the ending empty triangles of $T_1$, $T_2$ and $T_3$, respectively.
Let $t_1^{S^5},t_2^{S^5},t_3^{S^5}$ be ending empty triangles of $S^5$.
We identify $t_1^{T_1}$ with $t_1^{S^5}$; $t_1^{T_2}$ with $t_2^{S^5}$ and $t_1^{T_3}$ with $t_3^{S^5}$. The matching triangular triangle is defined to be $T_1\cup S^5\cup T_2\cup T_3$.
The matching triangular triangle is depicted in Figure~\ref{fig:matchtrian}.

\begin{proposition}
\label{prop:trianglematching}
The matching triangular triangle has exactly one perfect matching $M^1$ and exactly one matching $M^0$ with defect $\{1,2,3,a,b,c,\alpha,\beta,\gamma\}$. It has no matching with defect $E$, where $\emptyset\neq E\subsetneq\{1,2,3,a,b,c,\alpha,\beta,\gamma\}$.
\end{proposition}
\begin{proof}
The perfect matching is $M^1:=M^1(S^5)\cup M^L(T_1)\cup M^L(T_2)\cup M^L(T_3)$.
The matching $M^0$ is $M^0(S^5)\cup M^R(T_1)\cup M^R(T_2)\cup M^R(T_3)$.

Any matching of the matching triangular triangle with defect $E\subset\{1,2,3,a,b,c,\alpha,\beta,\gamma\}$ contains $M^1(S^5)$ or $M^0(S^5)$. This determines remaining triangles in a matching with defect $E\subseteq\{1,2,3,a,b,c,\alpha,\beta,\gamma\}$. Hence, there are just two matchings $M^1$ and $M^0$ with defect $E\subseteq\{1,2,3,a,b,c,\alpha,\beta,\gamma\}$.
\end{proof}

\begin{proposition}
\label{prop:mtttrip}
Matching triangular triangle $T$ is tripartite and there is a tripartition of $T$ such that $a,b,c\in E_1$; $1,2,3\in E_2$; $\alpha,\beta,\gamma\in E_3$.
\end{proposition}
\begin{proof}
Follows from Figure~\ref{fig:s5colored} and Figure~\ref{fig:tunneltrip}.
\end{proof}

\subsection{Linking three triangles by matching triangular triangle}
Let $\Delta$ be a triangular configuration. Let $t_1,t_2$ and $t_3$ be three edge disjoint triangles of $\Delta$.

The {\em link by matching triangular triangle} between $t_1,t_2$ and $t_3$ in $\Delta$ is the triangular configuration $\Delta'$ defined as follows.
Let $T$ be a matching triangular triangle defined in Section~\ref{sec:matchtriangle}.
Let $\{a,b,c\},\{1,2,3\},\{\alpha,\beta,\gamma\}$ be ending empty triangles of $T$.
Let $t_1^1,t_1^2,t_1^3$ and $t_2^1,t_2^2,t_2^3$ and $t_3^1,t_3^2,t_3^3$ be edges of $t_1$ and $t_2$ and $t_3$, respectively.
We relabel edges of $T$ such that $\left\{a,b,c\right\}=\left\{t_1^1,t_1^2,t_1^3\right\}$ and $\left\{1,2,3\right\}=\left\{t_2^1,t_2^2,t_2^3\right\}$ and $\left\{\alpha,\beta,\gamma\right\}=\left\{t_3^1,t_3^2,t_3^3\right\}$. We let $\Delta':=\Delta\cup T$.

\subsection{Construction}

Let $\Delta$ be a triangular configuration and let $w:T(\Delta)\mapsto\mathbb{R}$ be weights of triangles. We construct a tripartite triangular configuration $\Delta'$ and weights $w':T(\Delta')\mapsto\mathbb{R}$ in two steps.
First step: We start with triangular configuration
$$\Delta'_1:=\Delta_1\cup\Delta_2\cup\Delta_3$$
where $\Delta_1,\Delta_2,\Delta_3$ are disjoint copies of $\Delta$.
Let $t$ be a triangle of $\Delta$. We denote the corresponding copies of $t$ in $\Delta_1,\Delta_2,\Delta_3$ by $t_1,t_2,t_3$, respectively.

Second step: For every triangle $t$ of $\Delta$, we link $t_1,t_2,t_3$ in $\Delta'_1$ by triangular matching triangle $T$. We denote this triangular matching triangle by $T_t$. Then we remove triangles $t_1,t_2,t_3$ from $\Delta'_1$. We choose a triangle $t'$ from $M^1(T_t)$ and set $w'(t'):=w(t)$. We set $w'(t'):=0$ for $t'\in T(T_t)\setminus\{t\}$. The resulting configuration is desired configuration $\Delta'$. 

\begin{proposition}
\label{prop.tripartite}
Triangular configuration $\Delta'$ is tripartite.
\end{proposition}
\begin{proof}
The triangular configuration $\Delta'$ is constructed from three disjoint triangular configurations $\Delta_1,\Delta_2,\Delta_3$.
From these configurations all triangles are removed. Hence, we can put edges $E(\Delta_i)$ to set $E_i$ for $i=1,2,3$.
The remainder of $\Delta'$ is formed by matching triangular triangles. Every matching triangular triangle connects edges of $\Delta_1,\Delta_2,\Delta_3$. By Proposition~\ref{prop:mtttrip} the matching triangular triangle is tripartite and its ends belong to different partities.

\end{proof}
We denote by $2^X$ the set of all subsets $X$.
We define a mapping $f:2^{T(\Delta)}\mapsto 2^{T(\Delta')}$ as: Let $S$ be a subset of $T(\Delta)$ then
$$f(S):=\{M^1(T_t)\vert t\in S\}\cup\{M^0(T_t)\vert t\in T(\Delta)\setminus S\}.$$

\begin{proposition}
\label{prop.mappingf}
The mapping $f$ is a bijection between the set of perfect matchings of $\Delta$ and the set of perfect matchings of $\Delta'$ and $w(M)=w'(f(M))$ for every $M\subseteq T(\Delta)$.
\end{proposition}
\begin{proof}
By definition, the mapping $f$ is an injection. By Proposition~\ref{prop:trianglematching}, every inner edge of $T_t$, $t\in T(\Delta)$, is covered by $f(S)$ for any subset $S$ of $T(\Delta)$. Let $M$ be a perfect matching of $\Delta$. We show that $f(M)$ is perfect matching of $\Delta'$. 
$$f(M)=\{M^1(T_t)\vert t\in M\}\cup\{M^0(T_t)\vert t\in T(\Delta)\setminus M\}.$$
Let $e$ be an edge of $\Delta$ and let $e_1,e_2,e_3$ be corresponding copies in $\Delta_1,\Delta_2,\Delta_3$.
Let $t_1,t_2,\dots,t_l$ be triangles incident with edge $e$ in $\Delta$. Let $t_k$ be the triangle from perfect matching $M$ incident with $e$.
By definition of $\Delta'$, the edges $e_1,e_2,e_3$ are incident only with triangles of $T_{t_i}$, $i=1,\dots,l$.
The edges $e_1,e_2,e_3$ are covered by $M^1(t_k)$. The edges of $T_{t_i}$, $i=1,\dots,l$,  $i\neq k$ are covered by $M^0(T_{t_i})$. Hence $f(M)$ is a perfect matching of $\Delta'$.

Let $M'$ be a perfect matching of $\Delta'$. By Proposition~\ref{prop:trianglematching}, $M'=\{M^1(T_t)\vert t\in S\}\cup\{M^0(T_t)\vert t\in T(\Delta)\setminus S\}$ for some set $S$. The set $S$ is a perfect matching of $\Delta$. Thus, the mapping $f$ is a bijection.

\end{proof}
\begin{corollary}
\label{cor.matchpolyn}
$P_\Delta(x)=P_{\Delta'}(x)$. \qed
\end{corollary}

\begin{proposition}
\label{prop.trip1}
Let $\Delta$ be a triangular configuration with weights $w:T(\Delta)\mapsto\mathbb{R}$. Then there exist a tripartite triangular configuration $\Delta'$ and weights $w':T(\Delta')\mapsto\mathbb{R}$ such that there is a bijection $f$ between the set of perfect matchings $\mathcal{P}(\Delta)$ and the set of perfect matchings of $\mathcal{P}(\Delta')$. Moreover, $w(M)=w'(f(M))$ for every $M\in\mathcal{P}(\Delta)$, and $P_\Delta(x)=P_{\Delta'}(x)$.
\end{proposition}
\begin{proof}
Follows directly from Propositions~\ref{prop.tripartite} and~\ref{prop.mappingf} and Corollary~\ref{cor.matchpolyn}.
\end{proof}

\begin{proposition}
\label{prop:per}
Let $\Delta$ be a tripartite triangular configuration with tripartition $E_1,E_2,E_3$ such that $\lvert E_1\rvert=\lvert E_2\rvert=\lvert E_3\rvert$ and let $A_\Delta(x)$ be its triadjacency matrix. Then $P_\Delta(x)=\Per(A_\Delta(x))$.
\end{proposition}
\begin{proof}
We have
$$\Per(A_\Delta(x))=\sum_{\sigma_1,\sigma_2\in S_n}\prod^n_{i=1}a_{i\sigma_1(i)\sigma_2(i)}.$$
Every perfect triangular matching between partities $E_1,E_2,E_3$ can be encoded by two permutations $\sigma_1,\sigma_2$ and vice versa.
If matching $M$ is a subset of $T(\Delta)$, then $$\prod^n_{i=1}a_{i\sigma_1(i)\sigma_2(i)}=\prod^n_{i=1}x^{w([i\sigma_1(i)\sigma_2(i)])}=x^{w(M)},$$
where $[ijk]$ denotes a triangle of $\Delta$ with edges $i,j,k$.
If $M$ is not a subset of $T(\Delta)$, then there is $i$ such that $a_{i\sigma_1(i)\sigma_2(i)}=0$. Hence $\prod^n_{i=1}x^{w([i\sigma_1(i)\sigma_2(i)])}=0$.
Therefore
$$\sum_{\sigma_1,\sigma_2\in S_n}\prod^n_{i=1}a_{i\sigma_1(i)\sigma_2(i)}=P_\Delta(x).$$
\end{proof}

\section{Kasteleyn 3-matrices}
\label{sec.3mat}
We first introduce a necessary condition for a 3-matrix to be Kasteleyn. Let $A$ be a $\lvert V_0\rvert\times\lvert V_1\rvert\times \lvert V_2\rvert$
non-negative 3-matrix, where $\lvert V_i\rvert= m, i=1,2,3$. We first define two bipartite graphs $G_1, G_2$ as follows.
We let, for $i=1,2$, $G_i^A=G_i= (V_0, V_i, E_i)$ where 
$$
E_1= \left\{\{a,b\}\vert a\in V_0, b\in V_1 \text{ and }A_{abc}\neq 0 \text{ for some }c\right\},
$$
and
$$
E_2= \left\{\{a,c\}\vert a\in V_0, c\in V_2 \text{ and }A_{abc}\neq 0 \text{ for some }b\right\}.
$$

\begin{theorem}
\label{thm.k1}
If $A$ is such that both $G^A_1,G^A_2$ are Pfaffian bipartite graphs then $A$ is Kasteleyn.
\end{theorem}
\begin{proof}
Let $M_i$ be the biadjacency matrix of $G_i$ and let $\sign_i:E(G^A_i)\mapsto\{-1,1\}$ be the signing of the entries of $M_i$ which defines matrix $M_i'$ such that $\Per(M_i)= \det(M_i')$. 
We define 3-matrix $A'$ by 
$$
A'_{abc}= \sign_1(\{a,b\})\sign_2(\{a,c\})A_{abc}.
$$
We have
$$
\det(A')= \sum_{\s_1}\sign(\s_1)\times\sum_{\s_2}\sign(\s_2)
\prod_j \sign_2(\{j,\s_2(j)\})[\sign_1(\{j,\s_1(j)\})A_{j\s_1(j)\s_2(j)}].
$$
By the construction of $\sign_2$ we have that for each $\s_2$ and each $\s_1$, if
$\prod_j A_{j\s_1(j)\s_2(j)}\neq 0$ then
$$
\sign(\s_2)\prod_j \sign_2(\{j,\s_2(j)\})= 1.
$$
Hence
\begin{align*}
\det(A')&= \sum_{\s_1}\sign(\s_1)\times\sum_{\s_2}\prod_j [\sign_1(\{j,\s_1(j)\})A_{j\s_1(j)\s_2(j)}]\\
&=\sum_{\s_2}\times\sum_{\s_1}\sign(\s_1)\prod_j \sign_1(\{j,\s_1(j)\})A_{j\s_1(j)\s_2(j)}.
\end{align*}
Analogously by the construction of $\sign_1$ we have that for each $\s_1$ and each $\s_2$, if
$\prod_j A_{j\s_1(j)\s_2(j)}\neq 0$ then
$$
\sign(\s_1)\prod_j \sign_1(\{j,\s_1(j)\})= 1.
$$
Hence
$$
\det(A')= \sum_{\s_1,\s_2}\prod_j A_{j\s_1(j)\s_2(j)}= \Per(A).
$$

\end{proof}

In the introduction we defined the triadjacency 3-matrix of a triangular configuration as the adjacency matrix of
the {\em edges} of the triangles. We also defined a {\em matching} of a triangular configuration as a set of edge-disjoint triangles. In this section we concentrate on the vertices rather than on the edges.

A triangular configuration $\Delta$ is {\em vertex-tripartite} if vertices of $\Delta$ can be divided into three disjoint sets $V_1,V_2,V_3$ such that every triangle of $\Delta$ contains one vertex from each set $V_1,V_2,V_3$. We call the sets $V_1,V_2,V_3$ vertex-tripartition of $\Delta$.

The {\em vertex-adjacency} 3-matrix $A(x)=\left(a_{ijk}\right)$ of a vertex-tripartite triangular configuration $\Delta$ with vertex-tripartition $V_1,V_2,V_3$ is defined as follows: We set
\begin{equation*}
a_{ijk}:=\begin{cases}
         x^{w(t)}& \text{if }i\in V_1,j\in V_2,k\in V_3\text{ forms a triangle }t\text{ with weight }w(t),\\
         0& \text{otherwise}.
        \end{cases}
\end{equation*}

We will need the following modification of the notion
of a matching.
%
A set of triangles of a triangular configuration is called {\em strong matching} if its triangles are mutually vertex-disjoint.

\begin{proof}[Proof of Theorem \ref{thm.main3}]
Let $M$ be a $n\times n$ matrix and let $G= (V_1, V_2, E)$ be the bipartite graph of its
non-zero entries. We have $\lvert V_1\rvert= \lvert V_2\rvert= n$.  We order vertices of each $V_i$, $i=1,2$ arbitrarily and let
$V_i= \{v(i,1), \ldots, v(i,n)\}$. Let $V'_i= \{v'(i,1), \ldots, v'(i,n)\}$ be disjoint copy of $V_i$, $i= 1,2$.

We next define three sets of vertices $W_1, W_2, W_0$ and system of triangles $\D(G)= \D$ so that each triangle intersects each $W_i$ in exactly one vertex. 
\begin{align*}
W_1&= V_1\cup V_1'\cup \{w(1,e)\vert e\in E\},\\
W_2&= V_2\cup V_2'\cup \{w(2,e)\vert e\in E\},\\
W_0&= \{w(0,e)\vert e\in E\}\cup \{w(0,i,j)\vert i=1,2;  j= 1, \ldots, n\},
\end{align*}
\begin{align*}
\D= &\cup_{e= ab\in E}\{(a,b,w(0,e)), (w(0,e),w(1,e),w(2,e))\}\cup\\
&\cup_{j=1}^n\{(w(0,1,j), v'(2,j), w(1,e))\vert v(1,j)\in e\}\cup\\
&\cup_{j=1}^n\{(w(0,2,j), v'(1,j), w(2,e))\vert v(2,j)\in e\}.
\end{align*}

We let $A$ be the vertex-adjacency 3-matrix of the triangular configuration $\T(G)= \T= (W_0, W_1,W_2, \D)$.
We first observe that both bipartite graphs $G_1, G_2$ of $A$ (introduced before Theorem \ref{thm.k1}) are planar;
let us consider only  $G_1$, the reasoning for $G_2$ is the same. First, vertices $v'(1,j)$ and $w(0,2,j)$ 
are connected only among themselves in $G_1$. Further, the component of $G_1$ containing vertex $v(1,j)$
contains also vertex $w(0,1,j)$ and consists of $deg_G(v(1,j))$ disjoint paths of length $3$ between these two vertices. Here $deg_G(v(1,j))$  denotes the degree of $v(1,j)$ in graph $G$, i.e., the number of edges of $G$ incident with $v(1,j)$. Thus, by Theorem \ref{thm.k1}, $A$ is Kasteleyn.

We next observe that Kasteleyn signing is trivial. Let $D_1$ be the orientation of $G_1$ in which each edge is directed from $W_0$ to $W_1$. In each planar drawing of $G_1$, each inner face has an odd number of edges directed in $D_1$ clockwise. This means that $D_1$ is a {\em Pfaffian orientation} of $G_1$, and  $\Per(A)= \det(A)$
(see e.g. Loebl~\cite{L} for basic facts on Pfaffian orientations and Pfaffian signings).

Finally there is a bijection between the perfect matchings of $G$ and the perfect strong matchings of $\T$:
 if $P\subset E$ is a perfect matching of $G$ then let 
$$
P(\T)= \{(a,b,w(0,e))\vert e=ab\in P\}.
$$
We observe that $P(\T)$ can be uniquely extended to a perfect strong matching of $\T$, namely by
the set of triples $S_1\cup S_2\cup S_3$ where

\begin{align*}
S_1&= \cup_{e\in E\setminus P}\{(w(0,e),w(1,e),w(2,e))\},\\
S_2&= \cup_{j=1}^n\{(w(0,1,j), v'(2,j), w(1,e))\vert v(1,j)\in e\in P\},\\
S_3&= \cup_{j=1}^n\{(w(0,2,j), v'(1,j), w(2,e))\vert v(2,j)\in e\in P\}.
\end{align*}

Set $S_1$ is inevitable in any perfect strong matching containing $P(\T)$ since the vertices $w(0,e); e\notin P$ 
must be covered. This immediately implies that sets $S_2, S_3$ are inevitable as well.

On the other hand, if $Q$ is a perfect strong matching of $\T$ then $Q$ contains $P(\T)$ for some perfect matching
$P$ of $G$.

\end{proof}

\section{Application to 3D dimer problem}
\label{s.3dimer}
Let $Q$ be cubic $n\times n\times n$ lattice. The dimer partition function of $Q$, which is equal to the generating
function of the perfect matchings of $Q$, can be identified (by Theorem \ref{thm.main3}) with the determinant
of the Kasteleyn vertex-adjacency matrix of triangular configuration $\T(Q)$.
Natural question arises whether this observation can be used to study the 3D dimer problem. 

 We first observe that the natural embedding of $Q$ in 3-space can be simply modified to yield an embedding of $\T(Q)$ in 3-space. This can perhaps best be understood by figures, see 
Figure \ref{fig.3dimer}; this figure depicts configuration $\T(Q)$ around vertex $v$ of $Q$ with neighbors $u_1, \ldots, u_6$.

\begin{figure}[h]
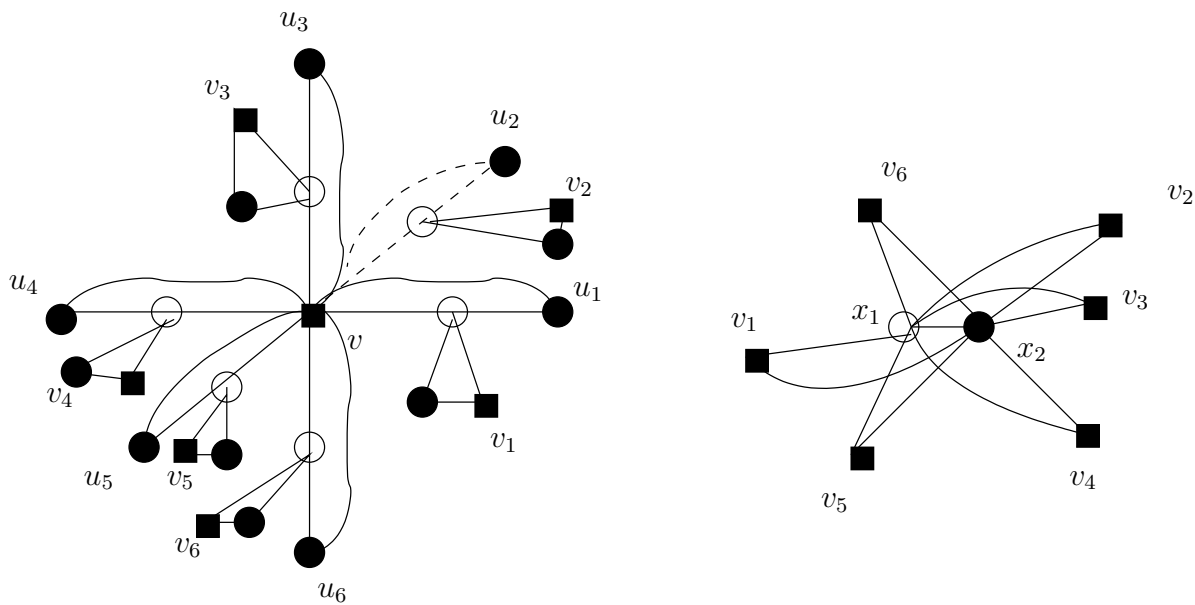

\begin{center}
\input 3dimer.pstex_t
\end{center}
\caption{Configuration $\T(Q)$ around vertex $v$ of $Q$ with neighbors $u_1, \ldots, u_6$ so that $u_1, u_3, u_4, u_6$ belong to the same plane in the 3-space, $u_2$ is 'behind' this plane and $u_5$ is 'in front of' this plane. Empty vertices belong to $W_0$, square vertices belong to $W_1$ and full vertices belong to $W_2$.}
\label{fig.3dimer}
\end{figure}

Triangular configuration  $\T(Q)$ is obtained by identification of vertices $v_i, i= 1,\ldots, 6$ in the left and right parts of Figure \ref{fig.3dimer}. Now assume that the embedding of left part of Figure \ref{fig.3dimer} is such that for each vertex $v$ of $Q$, the vertices $v_1, \ldots, v_6$ belong to the same plane and the convex closure of $v_1, \ldots, v_6$ intersects the rest of the configuration only in $v_1, \ldots, v_6$. Then we add the embedding of the right part, for each vertex $v$ of $Q$, so that
$x_1$ belongs to the plane of the $v_i$'s and $x_2$ is very near to $x_1$ but outside of this plane.

Summarizing, the dimer partition function of a finite 3-dimensional cubic lattice $Q$ may be written as the determinant of the vertex-adjacency 3-matrix of  triangular configuration   $\T(Q)$ which preserves the natural embedding of the cubic lattice. 

Calculating the determinant of a 3-matrix is hard, but perhaps formulas
for the determinant of the particular vertex-adjacency 3-matrix of $\T(Q)$, illuminating
the 3-dimensional dimer problem, may be found. An example of a formula valid for the determinant of a 3-matrix
is shown in the next subsection. It is new as far as we know but its proof is basically identical to the proof of Lemma 3.3 of Barvinok~\cite{B}.

\subsection{Binet-Cauchy formula for the determinant of 3-matrices}
\label{sub.fd}

We recall from the introduction that the determinant of a $n\times n\times n$ 3-matrix $A$ is defined to be
\begin{equation*}
\det(A)=\sum_{\s_1,\s_2\in S_n}\sign(\s_1)\sign(\s_2)\prod^n_{i=1}a_{i\sigma_1(i)\sigma_2(i)}.
\end{equation*} 

The next formula is a generalization of Binet-Cauchy formula (see the proof of Lemma 3.3 in Barvinok~\cite{B}).

\begin{lemma}
\label{l.bcf}
Let $A^1, A^2, A^3$ be real $r\times n$ matrices, $r\leq n$. For a subset $I\subset \{1, \ldots, n\}$ of 
cardinality $r$ we denote by $A^s_I$ the $r\times r$ submatrix of the matrix $A^s$ consisting of the columns of $A^s$ indexed by the elements of the set $I$. Let $C$ be the 3-matrix defined, for all $i_1, i_2, i_3$  by
$$
C_{i_1, i_2, i_3}= \sum_{j= 1}^n A^1_{i_1,j}A^2_{i_2,j}A^3_{i_3,j}.
$$
Then 
$$
\det(C)= \sum_I \Per(A^1_I)\det(A^2_I) \det(A^3_I),
$$ 
where the sum is over all subsets $I\subset \{1, \ldots, n\}$ of 
cardinality $r$
\end{lemma}
\begin{proof}
\begin{align*}
\det(C) &= \sum_{\s_1,\s_2\in S_r}\sign(\s_1)\sign(\s_2)\prod^r_{i=1}
\sum_{j=1}^nA^1_{i,j}A^2_{\s_1(i),j}A^3_{\s_2(i),j}\\
&=\sum_{\s_1,\s_2\in S_r}\sign(\s_1)\sign(\s_2)\times \sum_{1\leq j_1,\ldots,j_r\leq n}
\prod^r_{i=1}A^1_{i,j_i}A^2_{\s_1(i),j_i}A^3_{\s_2(i),j_i}\\
&=\sum_{1\leq j_1,\ldots,j_r\leq n}\sum_{\s_1,\s_2\in S_r}\sign(\s_1)\sign(\s_2)\times 
\prod^r_{i=1}A^1_{i,j_i}A^2_{\s_1(i),j_i}A^3_{\s_2(i),j_i}.
\end{align*}
Now, for all $J= (j_1,\ldots, j_r)$ we have 
\begin{align*}
&\sum_{\s_1,\s_2\in S_r}\sign(\s_1)\sign(\s_2)\times 
\prod^r_{i=1}A^1_{i,j_i}A^2_{\s_1(i),j_i}A^3_{\s_2(i),j_i}\\
=&(\prod_{i=1}^rA^1_{i,j_i})(\sum_{\s_1}\sign(\s_1)\prod_{i=1}^rA^2_{\s_1(i),j_i})
(\sum_{\s_2}\sign(\s_2)\prod_{i=1}^rA^3_{\s_2(i),j_i})\\
=&(\prod_{i=1}^rA^1_{i,j_i})\det(\tilde A^2_J)\det(\tilde A^3_J),
\end{align*}
where $\tilde A^s_J$ denotes the $r\times r$ matrix whose $i$th column is the $j_i$th column of matrix $A^s$.

If sequence $J$ contains a pair of equal numbers then the corresponding summand is zero,
since $\det(\tilde A^2_J)$ is zero. Moreover, if $J$ is a permutation, and $J'$ is obtained from $J$
by a transposition, then 
$$
\det(\tilde A^2_J)\det(\tilde A^3_J)= \det(\tilde A^2_{J'})\det(\tilde A^3_{J'}).
$$

 Therefore Lemma \ref{l.bcf} follows.

\end{proof}

\ifx\undefined\bysame
	\newcommand{\bysame}{\leavevmode\hbox
to3em{\hrulefill}\,}
\fi

\end{document}